\pdfoutput=1
\documentclass[11pt]{extarticle}

\usepackage[left=1in,right=1in,top=1in,bottom=1.3in]{geometry}
\usepackage{amsmath, amssymb, amsthm, latexsym, indentfirst, amsfonts, xcolor, mathtools, manfnt, subcaption} 
\usepackage[shortlabels]{enumitem}
\usepackage[utf8]{inputenc} 
\usepackage[english]{babel}

\usepackage[breaklinks,backref=page]{hyperref}
\usepackage{cleveref}
\hypersetup{
	colorlinks = true, 
	urlcolor = cyan, 
	linkcolor = teal, 
	citecolor = cyan 
}

\usepackage{tikz} 

\renewcommand*{\backref}[1]{}
\renewcommand*{\backrefalt}[4]{
	\ifcase #1 \textcolor{red}{Not cited.}%
	\or $\uparrow$#2%
	\else $\uparrow$#2%
	\fi%
}
\let\OLDthebibliography\thebibliography
\renewcommand\thebibliography[1]{
	\OLDthebibliography{#1}
	\setlength{\parskip}{0pt}
	\setlength{\itemsep}{0pt plus 0.3ex}
}

\newtheorem{Theorem}{Theorem}

\newtheorem{Lemma}{Lemma}
\newtheorem{Problem}{Problem}

\newcommand{\R}{\mathbb R}

\newcommand{\N}{\mathbb N}

\newcommand{\CC}{\mathcal{C}}

\newcommand{\x}{\mathbf x}

\newcommand{\z}{\mathbf z}

\newcommand{\ind}[1]{\smash{\xrightarrow[\smash{\raisebox{0.7ex}{\tiny $ind$}}]{#1}}}

\newcommand{\combine}[1]{{\,\oplus_{#1}\, }}

\title{Any 2-coloring of the plane contains monochromatic unit rhombuses}
\author{Kenneth Moore\thanks{HUN-REN Alfr\'ed R\'enyi Institute of Mathematics, Budapest, Hungary, \href{mailto:kjmoore@renyi.hu}{\tt kjmoore@renyi.hu}. Supported by ERC Advanced Grants ``GeoScape", no. 882971 and ``ERMiD", no. 101054936.}
\and Arsenii Sagdeev\thanks{Karlsruhe Institute of Technology, Karlsruhe, Germany, \href{mailto:sagdeevarsenii@gmail.com}{\tt sagdeevarsenii@gmail.com}.}}
\date{}

\begin{document}

\maketitle

\begin{abstract}
    In this note, we prove that any 2-coloring of the plane contains
    4 points of the same color forming a rhombus with unit sides and non-unit diagonals,
    answering a question of Axenovich, Liu, and the second author.
\end{abstract}

\section{Introduction}
An important branch of geometric Ramsey problems was founded by Erd\H{o}s, Graham, Montgomery, Rothschild, Spencer, and Strauss in three seminal papers \cite{ErdEtAl, ErdEtAla, ErdEtAlb}. Problems in this area concern partitions of Euclidean space into color classes, and which finite configurations of points are always found in one color class. Often, the configurations studied are congruent copies of a single set. The configurations studied by Axenovich, Liu, and the second author in \cite{AxeLiuSag} instead are unit-distance graphs, which may form families of point-sets with more than a single congruence class. Unit-distance graphs are well studied in several contexts, one of which is the problem of determining their maximum density; see \cite{EngHamSuVarZsa} for a relevant example.

For a graph $H = (V, E)$, we say that a set $A \subseteq \R^n$ with $|A| = |V|$ is a \emph{unit-copy} of $H$ if there is a bijection $f : V \to A$ such that $uv \in E$ implies $\|f (u) - f (v)\| = 1$. A unit-copy of $H$ is \emph{induced} if for any $uv\not \in E$, $\|f(u)-f(v)\|\neq 1$. We call graphs that have an induced unit-copy in $\R^n$ the \textit{unit-distance graphs in $\R^n$}. Moreover, a given set of points $A\subset\R^n$ is an induced unit-copy of a unit-distance graph, denoted $U(A)$, which is unique up to isomorphism.

When $A\subseteq \R^n$ and $\mathcal{F}$ is a family of subsets of $\R^n$, we use the arrow notation $A \xrightarrow[]{r} \mathcal{F}$ in place of the statement that any $r$-coloring of the points of $A$ contains a monochromatic congruent copy of an element of $\mathcal{F}$. Similarly, if $G$ is a graph and $\mathcal{H}$ is a family of graphs, $G \xrightarrow[]{r} \mathcal{H}$ (resp. $G \ind{r} \mathcal{H}$) denotes the statement that any $r$-coloring of the vertices of $G$ contains a monochromatic subgraph (resp. induced subgraph) that is a copy of some element of $\mathcal{H}$.

Of particular interest in \cite{AxeLiuSag} is the 4-cycle graph $C_4$ and its induced unit-copies. We denote by $\CC_k$ the family of all point sets $A\subset\R^2$ where $U(A)$ is the $k$-cycle $C_k$. Therefore, the sets in $\CC_4$ are exactly the 4-point rhombuses of side length 1 where neither diagonal has length 1. The following theorem answers Question~1 in \cite{AxeLiuSag}. 

\begin{Theorem}
\label{thm:main_c4_result}
$\R^2 \xrightarrow[]{2} \CC_4$.
\end{Theorem}

For more of an introduction to Euclidean Ramsey theory and several known results, see~\cite{AxeLiuSag} and the paper~\cite{CurMooYip} by Currier, the first author, and Yip. Our solution to Theorem~\ref{thm:main_c4_result} is a streamlined argument combining the ideas from both of these papers, although the additional complexity of our construction will necessitate computer assistance. 

\section{Proof of Theorem~\ref{thm:main_c4_result}}

If there is a 2-coloring of $\R^2$ with no monochromatic induced unit-copy of $C_4$, we argue that further patterns must be avoidable. The proof is carried out by providing a short sequence of such constraints over two Lemmas. This sequence is summarized informally in the following steps, and the proof is then given at the end of this section.
\vspace{3mm}

\noindent
If a 2-coloring of $\R^2$ has no monochromatic element of $\CC_4$, then
\begin{enumerate}[]
    \item we can also assume that no element of $\mathcal C_3$ is monochromatic (Lemma~\ref{lem:equilateral_triangle}), which implies
    \item no triple with distances $1,1,\frac16(3 + \sqrt{33})$ is monochromatic (Lemma~\ref{lem:computational}(a)), which implies
    \item every two points at distance $\frac{4}{\sqrt{3}}$ are of the same color (Lemma~\ref{lem:computational}(b)).
\end{enumerate}

The final statement is contradictory. One way to see this easily is to consider Figure~\ref{fig:contradiction}, which shows a set of seven points $B_7$ made up of a unit-rhombus and three other points such that all six dotted segments are of length $\frac{4}{\sqrt{3}}$. If the points opposite each other on each dotted segment are always the same color, the rhombus is forced to be all one color.

\begin{figure}[h]
    \centering
    \includegraphics[scale=0.6]{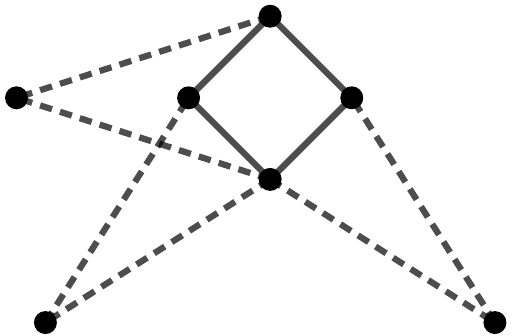}
    \caption{The set $B_7$, seven points with no valid coloring}
    \label{fig:contradiction}
\end{figure}

We prove the first step with Lemma~\ref{lem:equilateral_triangle} below, and because of the geometric connection from Lemma~\ref{lem:product}, this step can be phrased purely in terms of graphs. For graphs $G$ and $H$, we define the \textit{Cartesian product} $G\square H$ as the graph with vertex set $V(G) \times V(H)$, where $(v_1, v_2)$, $(u_1, u_2)$ form an
edge in $G\square H$ if and only if either $v_1u_1\in E(G)$ and $v_2 =u_2$, or $v_1=u_1$ and $v_2u_2\in E(H)$. We define the \textit{Cartesian power} recursively using the notation $G^{\square 1}=G$ and $G^{\square m} = G\square G^{\square m-1}$ for $m\geq 2$.

\begin{Lemma}[{Horvat--Pisanski~\cite[Theorem~3.4]{horvat2010products}}]
    \label{lem:product}
    If $G$ is a unit-distance graph in the plane, then $G^{\square m}$ is a unit-distance graph in the plane for all $m \in \N$.
\end{Lemma}

\begin{Lemma}
    \label{lem:equilateral_triangle}
    If $G$ is a graph where $G \xrightarrow{2} \{C_3, C_4\}$, then $G^{\square m} \ind{2} C_4$ for a sufficiently large $m \in \N$.
\end{Lemma}

\begin{proof}
    Similar ideas were used in~\cite[Section~2]{AxeLiuSag} to prove related properties of Cartesian powers.

    Take $m=|V(G)|$. Fix any red-blue coloring of the vertices of $G^{\square m}$.
    Consider \textit{slices} of $G^{\square m}$, i.e., subgraphs induced by the Cartesian products of $V(G)$ and $m-1$ singletons in any order; the \textit{direction} of a slice is the index of the position of $V(G)$ in the corresponding Cartesian product.
    
    Note that there are exactly $k \coloneqq m|V(G)|^{m-1}$ slices in $G^{\square m}$, each of which is a copy of $G$. Hence, each slice contains either a monochromatic triangle or a monochromatic induced 4-cycle (here we use the fact that each non-induced 4-cycle contains a triangle). If the latter holds at least for one slice, we are done since this gives the desired monochromatic induced 4-cycle in $G^{\square m}$. Hence, we can assume without loss of generality that each of the $k$ slices contains a monochromatic triangle.
    
    Pick a collection of $k$ such triangles, one per slice (for those slices that contain many monochromatic triangles, we pick one of them arbitrarily). By the pigeonhole principle, some 2 of these triangles intersect since $3k > |V(G)|^m$ by the choice of $m$. Note that these 2 triangles are of different directions and thus they share exactly 1 vertex. Label their vertices by $\x_{1,1},\x_{2,1},\x_{3,1}$ and $\x_{1,1},\x_{1,2},\z_{1,3}$, respectively. Note that all 5 of these points are of the same color, say red.
    
    Since the other cases are analogous, we can assume without loss of generality that the directions of the slices containing the triangles $\x_{1,1}\x_{2,1}\x_{3,1}$ and $\x_{1,1}\x_{1,2}\z_{1,3}$ are $1$ and $2$, respectively. This implies that for some triangles $y_1y_2y_3$, $z_1z_2z_3$ in $G$ and some $*\in V(G)^{m-2}$, we have $\x_{i,j}=(y_i,z_j,*)$ for all $i=1,\, j \in \{1,2,3\}$ and $i \in \{1,2,3\},\, j=1$ (note that the 3-elements sets $\{y_1,y_2,y_3\}$ and $\{z_1,z_2,z_3\}$ are not necessarily disjoint, and may even coincide).
    
    Consider 4 auxiliary vertices $\x_{i,j}$ defined analogously for $i,j\in \{2,3\}$, namely $\x_{i,j}=(y_i,z_j,*)$. If at least 1 of these 4 auxiliary vertices is red, say $\x_{i,j}$, then $\x_{1,1}\x_{i,1}\x_{i,j}\x_{1,j}$ is the desired red induced 4-cycle in $G^{\square m}$. Otherwise, if all 4 of them are blue, then $\x_{2,2}\x_{3,2}\x_{3,3}\x_{2,3}$ is the desired blue induced 4-cycle in $G^{\square m}$.
\end{proof}

Next, for two finite sets $A,B\subset \R^n$, we use $A \combine{m} B$ to denote a new set which is obtained as follows. For every ordered $m$-tuple of points in $A$ which is congruent with the first $m$ vertices in $B$, take a congruent copy of $B$ such that these two $m$-tuples coincide. The union of all such transformed copies of $B$ along with $A$ itself is $A \combine{m} B$. 

\begin{figure}[ht]
\centering
\includegraphics[scale=0.58]{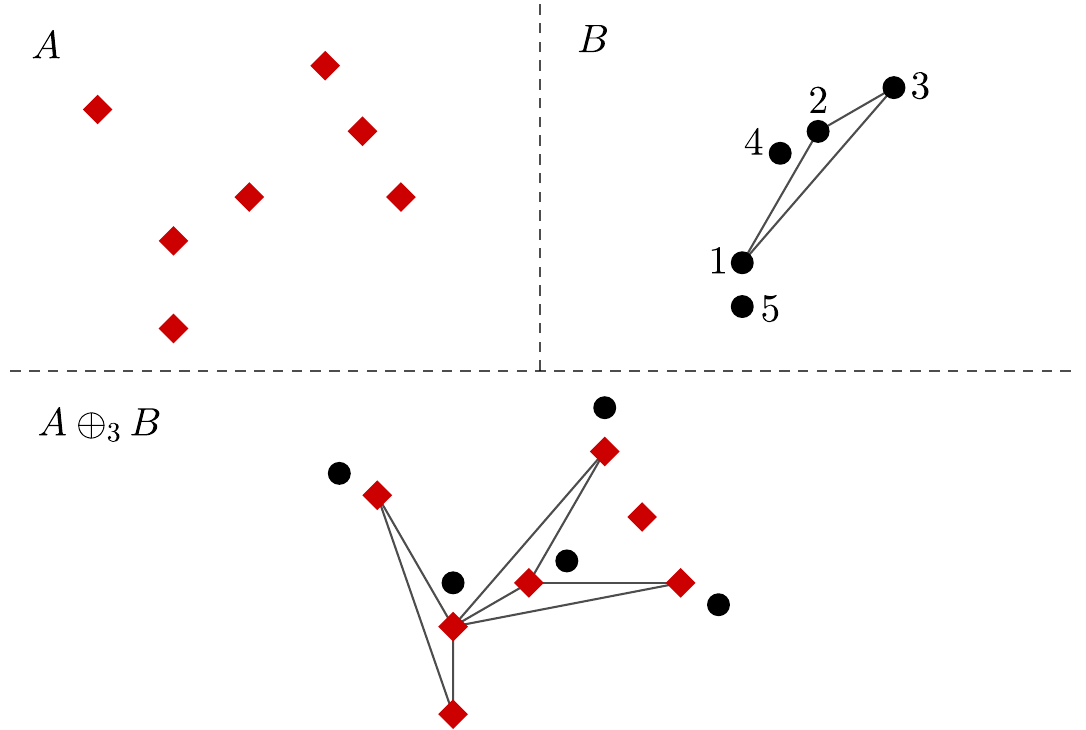}
\caption{Two sets combined with the $A\oplus_3B$ notation}
\label{fig:combine}
\end{figure}

\begin{Lemma}
    \label{lem:computational}
    Let $T$ be the triangle with sides $1,1,\frac16(3+\sqrt{33})$, and let $A\subset\R^2$ be a finite set. There are finite sets $B_{154}$ and $B_{46}$ such that 
    \begin{enumerate}[(a)]
        \item If $A\xrightarrow{2} \{T\} \cup \CC_3 \cup \CC_4$, then $A \combine{3} B_{154}\xrightarrow{2} \CC_3 \cup \CC_4$.
        \item If any 2-coloring of $A$ contains either two points at distance $\frac{4}{\sqrt3}$ that receive opposite colors or a monochromatic element of $\CC_4$, then $A\combine{2} B_{46} \xrightarrow{2} \{T\} \cup \CC_3 \cup \CC_4$.
    \end{enumerate}
\end{Lemma}

\begin{proof}
    We present a set of 154 points $B_{154}$, and a subset of 46 points $B_{46}\subset B_{154}$, which are depicted in Figure~\ref{fig:rule_constructions}. We prove that if the first three points of $B_{154}$ which form a copy of $T$ (shown in red in Figure~\ref{fig:154_construction}) are the same color, then the remaining points in $B_{154}$ cannot be 2-colored without creating a monochromatic element of $\CC_3\cup\CC_4$. We then show that if the first two points of $B_{46}$ which are at distance $\frac{4}{\sqrt3}$ have opposite colors (as shown in Figure~\ref{fig:46_construction}), the remaining points of $B_{46}$ cannot be 2-colored without creating a monochromatic congruent copy of $T$ or a monochromatic element of $\CC_3\cup\CC_4$.

    This is sufficient to prove the lemma, since $A \combine{3} B_{154}$ for example contains $A$. When 2-colored, if there is a monochromatic element of $\CC_3\cup\CC_4$, we are already done. Otherwise, there is a monochromatic copy of $T$ in $A$ and a congruent copy of $B_{154}$ in $A \combine{3} B_{154}$ with that copy of $T$ forming its first three points. Now this copy of $B_{154}$ contains the desired monochromatic element of $\CC_3\cup\CC_4$.

\begin{figure}[ht]
\centering
\begin{subfigure}{.5\textwidth}
  \centering
  \includegraphics[scale=0.45]{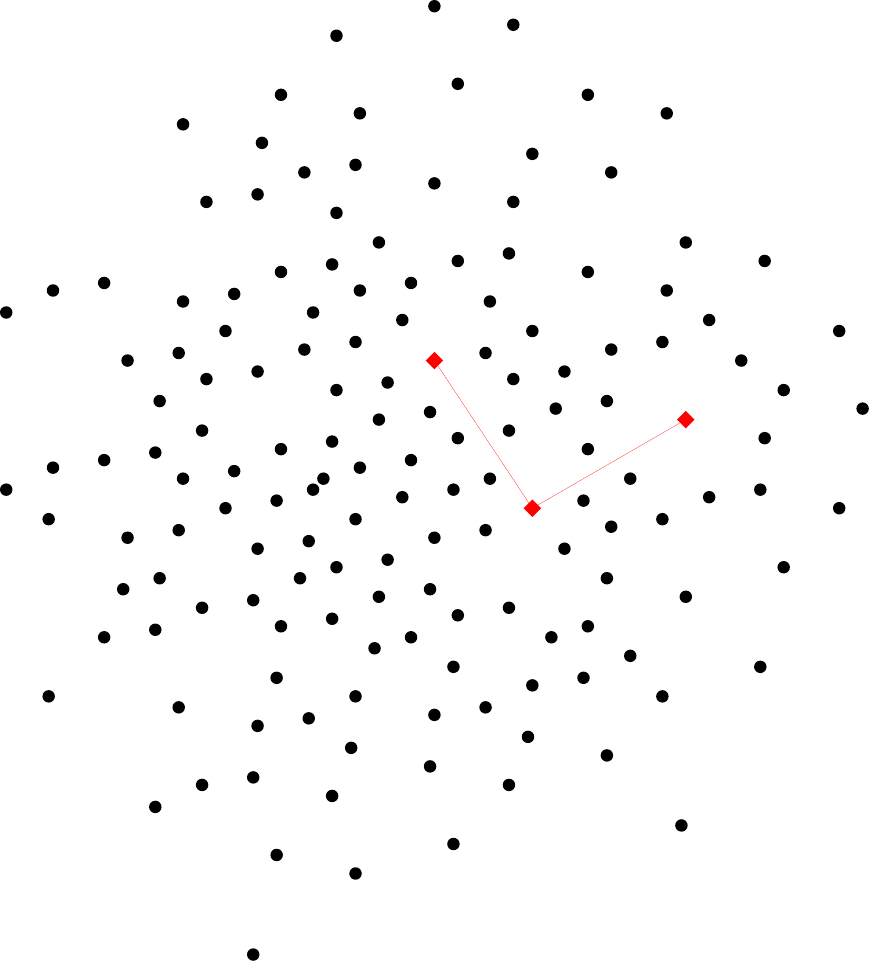}
    \caption{$B_{154}$ with three red points forming a $T$}
    \label{fig:154_construction}
\end{subfigure}%
\begin{subfigure}{.5\textwidth}
  \centering
   \raisebox{45pt}{\includegraphics[scale=0.45]{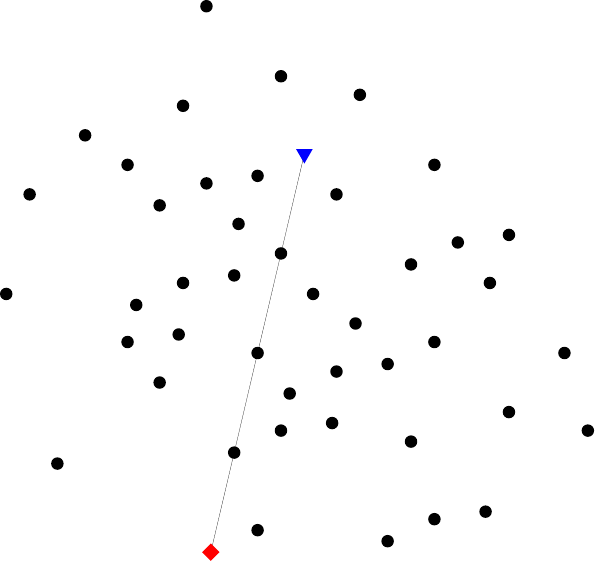}}
    \caption{$B_{46}$ with a red and a blue point at distance $\frac{4}{\sqrt{3}}$}
    \label{fig:46_construction}
\end{subfigure}
\caption{The constructions $B_{154}$ and $B_{46}$}
\label{fig:rule_constructions}
\end{figure}

The proof of this is computer assisted. We created a resources webpage \cite{resources}, where we keep text files containing coordinates of $B_{154}$ and $B_{46}$ (which can also be found in Appendix~\ref{apx:coordinates}), and a simple Python program. 
    
The Python program first converts a specified point set into a hypergraph, where the vertices corresponding to a set of points are connected by an edge if they geometrically form one of the constraints. Next, the program allows the user to assign the color $r$ to the first three vertices when testing $B_{154}$, or $r$ and $b$ to the first two vertices when testing $B_{46}$. It then attempts to find a valid $\{r,b\}$-coloring of the remaining vertices in the hypergraph. The program uses the SAT solver PySAT~\cite{SAT} in order to search for these colorings. In the cases of $B_{154}$ and $B_{46}$, it will report that no valid coloring exists, proving the lemma. 
    
The program can be used to search for colorings of any given set. In particular, one can delete any point from the file containing $B_{154}$ or $B_{46}$, and notice that the program finds a valid coloring.
    
A key fact that allows for this computational proof is that all points have the form
\begin{equation*}
\label{eq:moser_lattice}
\frac{1}{12}\left(\sqrt{3}a+\sqrt{11}b,c+\sqrt{33}d\right),
\end{equation*}

where $a,b,c,d$ are integers. Thus, if two points $p,q$ are represented by integers $[a_1,b_1,c_1,d_1]$ and $[a_2,b_2,c_2,d_2]$ such that
	\begin{align*}
	    3(a_1-a_2)^2 + 11(b_1-b_2)^2 + (c_1-c_2)^2 + 33 (d_1-d_2)^2  & \\
        +2\sqrt{33}\left((a_1-a_2)(b_1-b_2)+(c_1-c_2)(d_1-d_2)\right)&=144\delta^2
	\end{align*}
then $\|p-q\|=\delta$ exactly. This way, when forming the hypergraph, we can rigorously check whether a set of points forms one of our constraints with only integer arithmetic.
\end{proof}

\begin{proof}[Proof of Theorem~\ref{thm:main_c4_result}]
The steps can be combined backwards to create a single large set which cannot be 2-colored without creating a monochromatic element of $\CC_4$.

By Lemma~\ref{lem:computational}(b), $B_7\oplus_2 B_{46}\xrightarrow{2}\{T\}\cup\CC_3\cup\CC_4$. Then the set $A'=(B_7\oplus_2 B_{46}) \oplus_3 B_{154} $ satisfies $A'\xrightarrow{2}\CC_3\cup\CC_4$ by Lemma~\ref{lem:computational}(a). Next use Lemmas~\ref{lem:product} and~\ref{lem:equilateral_triangle}, and define the set $A$ satisfying $A \xrightarrow[]{2} \CC_4$ as an induced unit-copy of the graph $U(A')^{\square |U(A')|}$. Note $A$ has an astronomical number of points.
\end{proof}

\section{Conclusions}

The set described in the proof of Theorem~\ref{thm:main_c4_result} is not likely to be optimal. However we did put resources into lowering the point counts in Lemma~\ref{lem:computational} in order to speed up the computational parts of the proof. To help facilitate future research into similar problems, we give a brief outline of the approach that led to our result.

Our method for finding $B_{154}$ and $B_{46}$ was to use fast heuristic algorithms for coloring hypergraphs, and search over point sets that could be used to prove new constraints. The `seed' point sets that worked best were created from sets that start with many unit distances. $B_{154}$ and $B_{46}$ are subsets of the set shown in Figure~\ref{fig:seed_construction}. To create this set, we begin with $A_{33}$ shown in red, which is a unit-copy of one of the densest known unit-distance graphs on $34$ points (from \cite{EngHamSuVarZsa}) with one point subtracted. The seed set utilized is $A_{258}\coloneqq (A_{33}\oplus_2 C_3)\oplus_2 C_3$, on 258 points.

\begin{figure}[ht]
    \centering
    \includegraphics[scale = 1.5]{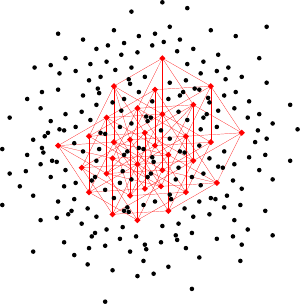}
    \caption{Construction of a good seed with 258 points}
    \label{fig:seed_construction}
\end{figure}

There are in fact many other 3-point sets in $A_{258}$ which cannot be monochromatic in any valid coloring, besides the triangle $T$ utilized in Lemma~\ref{lem:computational}(a). For each of these, we used another heuristic algorithm to quickly check how many points could be discarded from $A_{258}$ while maintaining this coloring property, and the lowest number obtained was 154 with the triangle $T$. Then we applied a similar procedure to pairs and obtained the number 46 with the segment of length $\frac{4}{\sqrt{3}}$.



It was shown in~\cite[Proposition~1.7]{AxeLiuSag} that $\R^2 \not\xrightarrow[]{4} \CC_4$. The case of 3 colors remains open.

\begin{Problem}
    Does $\R^2 \xrightarrow[]{3} \CC_4$?
\end{Problem}

There are several questions posed in \cite{AxeLiuSag}. We repeat two of those on longer cycles here as they are the most related and perhaps most susceptible to the methods used in this paper. The following is Question 2 and part of Question 3 in \cite{AxeLiuSag}.

\begin{Problem} Is it true that $\R^2\xrightarrow{3} \CC_6$ and $\R^2\xrightarrow{3} \CC_{10}$? Does $\R^2\xrightarrow{2} \CC_{5}$?
\end{Problem}

The approach used here is also similar to the one used in \cite{CurMooYip} to show that $\R^2\xrightarrow{2} \ell_3$, where $\ell_m\coloneqq \{(0,0),(0,1),...\, ,(0,m-1)\}$. But both in that paper and in this one, few steps are taken and each step is uncomplicated relative to the speed at which these colorings can be checked. Perhaps with a longer sequence of constraints, one could prove the following.

\begin{Problem}
Determine whether $\R^2\xrightarrow{2} \ell_4$.
\end{Problem}

\noindent
{\bf Acknowledgments.} The second author thanks Maria Axenovich, N\'ora Frankl, and Dingyuan Liu for useful discussions at the beginning of this project.

\bibliographystyle{abbrv}
\bibliography{bibliography} 
\appendix

\section{Point coordinates}
\label{apx:coordinates}

\noindent
The tuples representing the points of $B_{154}$:
\begin{align*}
    & [-1, 3, -9, 1], [5, 3, -3, 1], [-1, 1, 1, 1], [-11, 1, 1, -1], [0, 4, -8, 0], [-6, 2, -4, 0], [-5, 3, -3, -1], \\
    & [-7, 3, -3, 1], [0, 2, -10, 0], [-1, 3, 3, 1], [-7, 1, -5, 1], [-6, 2, -16, 0], [-11, 3, 3, -1], [-6, 0, 6, 0], \\
    & [5, 1, -5, 1], [-6, 2, 8, 0], [-6, 0, -6, 0], [0, 2, 2, 0], [-7, 3, -15, 1], [-5, 1, 7, -1], [-1, 1, 13, 1], \\
    & [-1, 1, -11, 1], [-5, 5, -13, -1], [5, 5, -13, 1], [0, 4, 4, 0], [0, 4, -20, 0], [-2, 4, -8, 2], [1, 1, -11, -1], \\
    & [6, 4, -2, 0], [6, 4, -14, 0], [-5, 3, 9, -1], [-4, 0, -6, -2], [-10, 2, 2, -2], [-5, 3, -15, -1], \\
    & [-10, 4, -8, -2], [1, 3, -9, -1], [-7, 3, 9, 1], [-2, 2, 2, 2], [-8, 0, -6, 2], [-12, 4, -8, 0], [-13, 3, -9, 1], \\
    & [0, 2, -22, 0], [5, 3, -15, 1], [6, 2, -4, 0], [6, 2, -16, 0], [-1, -1, -13, 1], [5, 3, 9, 1], [-1, 5, -7, 1], \\
    & [4, 2, 8, 2], [-1, 3, 15, 1], [4, 4, -2, 2], [-2, 0, 0, 2], [1, 3, 3, -1], [4, 2, -4, 2], [-1, 3, -21, 1], \\
    & [-1, 5, -19, 1], [4, 4, -14, 2], [-6, 4, -14, 0], [-2, 0, -12, 2], [-13, 1, 1, 1], [-8, 4, -2, 2], \\
    & [-2, 2, -10, 2], [-7, -1, 5, 1], [-7, 1, -17, 1], [-6, -2, -8, 0], [-13, 1, -11, 1], [-12, 0, 12, 0], \\
    & [-6, -2, 4, 0], [-7, 1, 19, 1], [-2, 0, 12, 2], [-7, -1, 17, 1], [-12, 2, -22, 0], [-11, 1, -11, -1], \\
    & [-6, 2, -28, 0], [-6, 4, -26, 0], [-11, 3, -21, -1], [-11, 5, -7, -1], [-11, 1, 13, -1], [-13, 3, 3, 1], \\
    & [-10, 0, 0, -2], [-17, 3, -3, -1], [-6, 4, -2, 0], [-11, 3, -9, -1], [-1, -1, 11, 1], [-6, 0, 18, 0], \\
    & [-4, 0, 6, -2], [-8, 0, 6, 2], [11, 1, 1, 1], [5, 1, -17, 1], [10, 2, -10, 2], [5, -1, 5, 1], [10, 0, 0, 2], \\
    & [-18, 2, -4, 0], [-18, 2, -16, 0], [-10, 2, -10, -2], [-13, 5, -7, 1], [-12, 2, 2, 0], [-12, 0, 0, 0], \\
    & [10, 2, 2, 2], [10, 4, -8, 2], [11, 3, -9, 1], [11, 3, 3, 1], [-6, 6, -12, 0], [0, 2, 14, 0], [-8, 2, 8, 2], \\
    & [-6, 2, 20, 0], [4, 0, 6, 2], [0, -2, -2, 0], [-12, 0, -12, 0], [-1, -1, -1, 1], [0, 0, -12, 0], [-6, 0, -18, 0], \\
    & [6, 2, 8, 0], [1, -1, -1, -1], [-8, 4, -14, 2], [-5, 1, -17, -1], [-12, -2, 10, 0], [-14, 0, 0, 2], \\
    & [-12, 2, -10, 0], [-5, -1, 17, -1], [1, 1, 1, -1], [-10, 0, 12, -2], [-5, 1, 19, -1], [1, 1, 13, -1], \\
    & [6, 0, 6, 0], [5, 1, 19, 1], [10, 0, 12, 2], [4, 0, 18, 2], [-1, 1, 25, 1], [-7, 1, 7, 1], [-1, -1, 23, 1], \\
    & [5, 1, 7, 1], [4, 2, -16, 2], [-1, 1, -23, 1], [4, 0, -6, 2], [-5, -1, 5, -1], [0, 0, 0, 0], [-10, 0, -12, -2], \\
    & [5, 5, -1, 1], [-2, 2, 14, 2], [-8, 2, -4, 2], [-14, 0, -12, 2], [1, 5, -19, -1], [-2, 2, -22, 2], [5, -1, -7, 1], \\
    & [-2, -2, 10, 2], [-7, -1, -7, 1], [-12, -2, -2, 0], [-10, 2, -22, -2], [0, -2, 10, 0], [6, 0, -6, 0], \\
    & [0, 0, 12, 0], [-5, 1, -5, -1], [1, -1, 11, -1] \\
\end{align*}

\noindent
The tuples representing the points of $B_{46}$:
\begin{align*}
    & [-4, 0, -6, -2], [-8, 4, -2, 2], [0, 4, -8, 0], [-6, 2, -4, 0], [-5, 3, -3, -1], [-7, 3, -3, 1], [0, 2, -10, 0], \\
    & [-1, 3, 3, 1], [-1, 3, -9, 1], [-7, 1, -5, 1], [-6, 2, -16, 0], [-11, 3, 3, -1], [-6, 0, 6, 0], [5, 1, -5, 1], \\
    & [-6, 2, 8, 0], [-1, 1, 1, 1], [-6, 0, -6, 0], [0, 2, 2, 0], [-7, 3, -15, 1], [-5, 1, 7, -1], [-1, 1, -11, 1], \\
    & [0, 4, 4, 0], [-2, 4, -8, 2], [1, 1, -11, -1], [1, 3, -9, -1], [-7, 3, 9, 1], [-2, 2, 2, 2], [-13, 3, -9, 1], \\
    & [5, 3, -15, 1], [6, 2, -4, 0], [-1, 3, -21, 1], [-13, 1, 1, 1], [-7, -1, 5, 1], [-13, 3, 3, 1], [-10, 0, 0, -2], \\
    & [-6, 4, -2, 0], [-12, 0, 0, 0], [-8, 2, 8, 2], [-1, -1, -1, 1], [1, -1, -1, -1], [1, 1, 1, -1], [-7, 1, 7, 1], \\
    & [-5, -1, 5, -1], [0, 0, 0, 0], [-8, 2, -4, 2], [-5, 1, -5, -1] \\
\end{align*}

\end{document}